\documentclass[graybox]{svmult}

\usepackage{mathptmx}       
\usepackage{helvet}         
\usepackage{courier}        
\usepackage{type1cm}        
\usepackage{makeidx}         
\usepackage{graphicx}        
\usepackage{multicol}        
\usepackage[bottom]{footmisc}

\usepackage{amsmath,amssymb}

\def\uOdt{u_{\mathcal{O},\delta t}}

\makeindex             

\begin{document}
\title*{Convergence of finite volume scheme\\ for degenerate parabolic problem\\ with zero flux boundary condition}
\titlerunning{Convergence of FV scheme for degenerate parabolic zero-flux problem}
\author{Boris Andreianov \and Mohamed Karimou Gazibo}
\institute{Boris Andreianov \at Laboratoire de Math\'{e}matiques CNRS UMR 6623, Universit\'{e} de Franche-Comte, 16 route de Gray, 25030 Besan{\c c}on, France\;\;\; and\;\;\; Institut F\"ur Mathematik, Technische Universit\"at Berlin, Stra{\ss}e des 17. Juni 136, 10623 Berlin, Germany;\;\; \email{bandreia@univ-fcomte.fr}
\and Mohamed Karimou Gazibo \at Laboratoire de Math\'{e}matiques CNRS  UMR 6623, Universit\'{e} de Franche-Comte, 16 route de Gray, 25030 Besan{\c c}on, France;\;\;\email{mgazibok@univ-fcomte.fr} }
%
%
\maketitle
\vspace*{-5pt}
\abstract*{This note is devoted to the study of the finite volume methods used in the discretization of degenerate parabolic-hyperbolic equation with zero-flux boundary condition. The  notion of an entropy-process solution, successfully used for the Dirichlet problem, is insufficient to obtain a uniqueness and convergence result because of a lack of regularity of solutions on the boundary. We infer the uniqueness of an entropy-process solution using the tool of the nonlinear semigroup theory by passing to the new abstract notion of integral-process solution. Then, we prove that numerical solution converges to the unique entropy solution as the mesh size tends to 0. }
\vspace*{-3pt}
\abstract{This note is devoted to the study of the finite volume methods used in the discretization of degenerate parabolic-hyperbolic equation with zero-flux boundary condition. The  notion of an entropy-process solution, successfully used for the Dirichlet problem, is insufficient to obtain a uniqueness and convergence result because of a lack of regularity of solutions on the boundary. We infer the uniqueness of an entropy-process solution using the tool of the nonlinear semigroup theory by passing to the new abstract notion of integral-process solution. Then, we prove that numerical solution converges to the unique entropy solution as the mesh size tends to 0. }
\section{Introduction}\label{sec:1}
Our goal is to study convergence of a finite volume scheme  for a  degenerate parabolic equation with zero-flux boundary condition in a regular bounded domain $\Omega\in\mathbb R^\ell$ arising, e.g., in sedimentation and traffic models:
\begin{equation}\label{eq:problemP}
\tag{\text{\rm P}}
\left\{\begin{array}{rll}
u_t+\mbox{div }f(u)-\Delta\phi(u)&=0         &\mbox{ in } \;\;\; Q=(0,T)\times\Omega,\\
                     u(0,x)&=u_0(x)   &\mbox{ in } \;\;\; \Omega,\\
(f(u)-\nabla\phi(u)).\eta&=0           &\mbox{ on } \;\;\; \Sigma=(0,T)\times\partial\Omega.
\end{array}\right.
\end{equation}
Here $\phi$  is a non-decreasing Lipschitz continuous function, moreover, there exists $u_c\in [0, u_{\max}]$ with $u_{\max}>0$ such that $\phi|_{[0,u_c]}\equiv 0$ but $\phi'|_{[u_c,u_{\max}]}>0$. The case $u_c=u_{max}$ was understood in \cite{BFK}. In the range $[0,u_c]$ of values of $u$, $(P)$ degenerates into a hyperbolic problem, and admissibility criteria of Kruzhkov type are needed to single out the unique and physically motivated weak solution (see, e.g., \cite{VAS,BFK}).
 We require that the flux function $f$ is Lipschitz,  genuinely nonlinear on $[0,u_c]$; moreover, $[0,u_{\max}]$ is an invariant domain for the evolution of $(P)$ due to assumption
\begin{equation}\label{fmax}\tag{H1}
f(0)=f(u_{\max})=0,\;\;\; u_0\in L^\infty(\Omega;[0,u_{\max}])
\end{equation}
(the latter means the space of measurable on $\Omega$ functions with values in $[0,u_{\max}]$).
In the work \cite{BG}, inspired by \cite{BFK} we proposed a new entropy formulation of $(P)$ saying that $u\in L^\infty(Q;[0,u_{\max}])$ is an entropy solution of $(P)$  if $u\in C([0,T];L^1(\Omega))$ with $u(0)=u_0$, $\phi(u)\in L^2(0,T;H^1(\Omega))$ and $\forall k\in [0,u_{\max}]$
\begin{equation}\label{ESP}
  |u-k|_t + \mbox{div }\bigl({sign}(u-k) \bigl[f(u)-f(k)-\nabla\phi(u)\bigr]\bigr)\;\leq\; |f(k).\eta|\,d{\mathcal{H}}^{\ell-1}
  \end{equation}
in $\mathcal D'((0,T)\times\overline{\Omega})$, where $\eta$ is the exterior unit normal vector to the boundary $\Sigma=(0,T)\times\partial\Omega$ and the last term is taken with respect to the Hausdorff measure ${\mathcal{H}}^{\ell-1}$ on $\Sigma$.  Contrary to the Dirichlet case (cf. \cite{
EGHMichel}) where the boundary condition is relaxed,
\eqref{ESP} implies that zero-flux condition in $(P)$ holds in the weak sense.

Existence of an entropy solution to $(P)$ can be obtained by standard vanishing viscosity method,
relying in particular on the \emph{strong compactness} arguments derived from genuine nonlinearity of $f|_{[0,u_c]}$ and non-degeneracy of $\phi|_{[u_c,u_{max}]}$, see~\cite{PAN-comp}. But in order to prove uniqueness,
 one faces a serious difficulty (not relevant in the case $u_c=u_{max}$, \cite{BFK}) related to the lack of regularity of the flux $\mathcal F[u]:= f(u)-\nabla\phi(u)$
and specifically, to the weak sense in which the normal component $\mathcal F[u].\eta$ of the flux annulates on $\Sigma$. Techniques of nonlinear semigroup theory (see, e.g., \cite{BCP,BarthelemyBenilan}) can be used to circumvent this regularity problem in some cases (see \cite{BF,BG}) and to prove well-posedness for $(P)$ in the sense \eqref{ESP}. Let us present the key arguments: indeed, they are also important for study of convergence of the Finite Volume scheme for $(P)$, which is the goal of this note.
The standard doubling of variables method based upon formulation \eqref{ESP} readily leads to the uniqueness and $L^1$ contraction property
  \begin{equation}\label{eq:L1contraction}
\forall t\in [0,T]\;\;\;    \|u(t,\cdot)-\hat u(t,\cdot)\|_{L^1}\leq \|u_0-\hat u_0\|_{L^1}
\end{equation}
if we compare two solutions $u,\hat u$ such that the strong (in the sense of $L^1$ convergence, see \cite{VAS,PAN1}) trace of the normal flux $\mathcal F[u].\eta$ at the boundary exists. In the sequel, we call such solutions \emph{trace-regular}.
Every entropy solution is a trace-regular in the case of the pure hyperbolic problem (case $u_c=u_{\max}$, see \cite{VAS,PAN1,BFK}).
The idea of symmetry breaking in the doubling of variables (see \cite{BF}) permits an extension
of \eqref{eq:L1contraction} to a kind of weak-strong comparison principle where $u$ is a general solution
and $\hat u$ is a trace-regular solution. When a sufficiently large family of trace-regular solutions
is available, uniqueness of a general solution and principle \eqref{eq:L1contraction} may follow by density arguments.
 A closely related technique consists in exploiting the above weak-strong comparison arguments using the idea of integral solution and somewhat stronger regularity properties of \emph{stationary solutions}. E.g.,
for the pure parabolic one ($u_c=0$, see \cite{BF}) every entropy solution of the stationary problem
\begin{equation}\label{eq:problemS}
\tag{\text{\rm S}}
\hat u+\mbox{div }f(\hat u)-\Delta\phi(\hat u)=g \;\text{ in $\Omega$}, \;\;
(f(\hat u)-\nabla\phi(\hat u)).\eta=0  \;\text{ on $\partial\Omega$}
\end{equation}
with $g\in L^\infty(\Omega)$ is trace-regular if $f\circ\phi^{-1}\in \mathbf{C}^{0,\gamma}$, $\gamma>0$ (see \cite{BF}). This observation, in conjunction with the use of integral solutions (\cite{BCP}) of abstract evolution problem
\begin{equation}\label{eq:Abstr-Pb}
   u' + Au \ni h, \;\;\;u(0)=u_0
\end{equation}
 for suitably defined operator $A=A_{f, \phi}$ (problem $(S)$ taking the form $(\text{Id}+A_{f,\phi})u\ni g$) permits to get uniqueness of entropy solution in \cite{BF}, for the parabolic case $u_c=0$.
 Let us stress that the question of uniqueness for $(P)$ with $u_c\notin \{0,u_{max}\}$ and $\ell>1$ remains open.
 The one-dimensional hyperbolic-parabolic case ($\ell=1$, $\Omega=(a,b)$ with arbitrary $ u_c \!\in[0, u_{\max}]$) has been treated by the authors in \cite{BG}, using the above abstract approach along with the elementary observation that yields trace-regularity:
\begin{equation*}\label{eq:trace-regularity-1D}
\bigl(f(\hat u)-\phi(\hat u)_x\bigr)_x=g-u\in L^\infty((a,b)) \;\;\Rightarrow\;\; \mathcal F[u]=\bigl(f(\hat u)-\phi(\hat u)_x\bigr) \in \mathbf{C}([a,b]).
\end{equation*}

Another essential aspect of the study of $(P)$ is to justify convergence of numerical approximations.
The difference with the existence proof is that, for numerical approximations, the use of \emph{strong compactness}
arguments is very technical, and \emph{weak compactness} methods are often preferred.
 Such study relying on \emph{nonlinear weak-$*$ compactness} technique of \cite{eymard2000finite,EGHMichel}  is our goal in this note. We study a finite volume scheme  discretization in the spirit of \cite{EGHMichel} for $(P)$ on a family of admissible meshes $(\mathcal O_h)_h$ with implicit time stepping. According to the standard weak compactness estimates, as for the Dirichlet problem (\cite{EGHMichel}) approximate solutions $u^h:=u_{\mathcal O_h,\delta t_h}$ converge up to a subsequence, as the discretization size $h$ goes to zero, towards an \emph{entropy-process solution} $\nu$. This notion closely related to  Young measures' techniques (see \cite{eymard2000finite} and references therein) incorporates dependence on an additional variable $\alpha\in[0,1]$ which may represent oscillations in the family $(u^h)_{h}$. It remains to prove the uniqueness of an entropy-process solution which implies the independence of $\nu(t,x,\alpha)$ on $\alpha$ so that $u(t,x)\equiv \nu(t,x,\alpha)$ is an entropy solution of $(P)$. As for the proof of uniqueness of an entropy solution discussed above, we face the major difficulty due to
the lack of regularity of $\mathcal F[u].\eta$.
Hence, we found it useful to define the  new notion of \emph{integral-process solution}
 in the framework of abstract problem \eqref{eq:Abstr-Pb}. 
Following the pattern of the uniqueness proofs in \cite{BF,BG}, we compare an entropy-process solution of $(P)$ and a trace regular solution of $(S)$, then we prove that an entropy-process solution of $(P)$ is an integral-process solution of \eqref{eq:Abstr-Pb} defined for an appropriate $m$-accretive operator $A_{f,\phi}$. The convergence result holds due to the fact that the integral-process solution coincides with the unique integral solution of \eqref{eq:Abstr-Pb};
and the latter one coincides with the unique entropy solution of $(P)$ in the sense \eqref{ESP}.

The remainder of this note is organized as follows. In Section~\ref{sec:2} we present our scheme. In Section~\ref{sec:3} we present the standard steps of convergence arguments for the problem $(P)$, obtained as for Dirichlet problem (\cite{EGHMichel}). In Section~\ref{sec:4}, we achieve the  convergence result using classical and new tools of the nonlinear semigroup theory. In Remark~\ref{rem:direct-IPS}, we sketch a convergence argument for Finite Volume schemes based upon a direct use of integral-process solutions, bypassing the entropy-process ones.

\vspace*{-5pt}\section{Description of the finite volume scheme for $(P)$}\label{sec:2}
Let us begin with considering an admissible mesh $\mathcal{O}$  of $\Omega$ (see \cite{eymard2000finite,EGHMichel}) for space discretization and using the conventional notation present in the main literature.
  Because we consider the zero-flux boundary condition, we don't need to distinguish between interior and exterior control volumes $K$, only inner interfaces $\sigma$ between volumes are needed in order to formulate the scheme.  
 For $K\in\mathcal{O}$ and $\sigma\in{\varepsilon}_K$, we denote by $\tau_{K,\sigma}$ the transmissivity coefficient. For the approximation of the convective term, we consider the numerical convection fluxes $F_{K,\sigma}:\mathbb R^2\longrightarrow\mathbb R$ that are consistent with $f$, monotone, Lipschitz regular, and conservative (see \cite{eymard2000finite,EGHMichel
  }). 
  
The values of the discrete unknowns $u_K^{n+1}$ for all control volume $K\in\mathcal O$, and $n\in\mathbb N$ are defined  thanks to the following relations:  first we initialize the scheme by
\begin{align}\label{esti0}
u_K^0=\frac{1}{m(K)}\displaystyle\int_K u_0(x)dx\;\;\;\forall K\in\mathcal O,
\end{align}
then,  we use the implicit scheme for the discretization of problem $(P)$:\\[1pt] $\forall n>0,\;\forall K\in\mathcal O$,
\begin{align}\label{esti1}
 m(K)\frac{u_K^{n+1}\!-\!u_K^n}{\delta t}+\displaystyle\sum_{\sigma\in\varepsilon_K}
 \Bigl( F_{K,\sigma}(u_{K}^{n+1}\!,u_{K,\sigma}^{n+1})
 -\tau_{K,\sigma}\bigl(\phi(u_{K,\sigma}^{n+1})\!-\!\phi(u_{K}^{n+1})\bigl) \Bigr)=0.
\end{align}
If the scheme has a solution $(u_{K}^n)_{K,n}$, we will say that the approximate solution to $(P)$ is the piecewise constant function $\uOdt(t,x)$ defined by:
\begin{equation}\label{eq:def-discrsol}
\uOdt(t,x)=u^{n+1}_K \mbox{ for } x\in K \mbox{ and } t\in (n\delta t,(n+1)\delta t].
\end{equation}
A weakly consistent discrete gradient $\nabla_{\!\mathcal O}\phi(\uOdt)$
is defined ``per diamond''; we refer to \cite{mathese} for details.
Let us stress that the zero-flux boundary condition is included in the scheme, since the flux terms on $\partial K\cap \partial \Omega$ are set to be zero in equations \eqref{esti1}.

\vspace*{-5pt}\section{Analysis of the approximate solution: classical arguments}\label{sec:3}
Following the guidelines of \cite{eymard2000finite,EGHMichel}, we can justify uniqueness of discrete solutions,
obtain several uniform estimates (confinement of values of $\uOdt$ in $[0,u_{max}]$, weak $BV$ estimate for
$\uOdt$, discrete $L^2(0,T;H^1(\Omega))$ estimate of $\phi(\uOdt)$), and derive existence of $\uOdt$. We refer to the PhD thesis \cite{mathese} of the second author for details, with a particular emphasis on the treatment of boundary volumes.
 It follows that the discrete solution $\uOdt$ satisfies the approximate continuous entropy formulation.
\begin{theorem}\label{th:approx-ineq}
Let $\uOdt$ be the approximate solution of the problem $(P)$ defined by \eqref{esti0},\eqref{esti1},\eqref{eq:def-discrsol}. Then the following approximate entropy inequalities hold:\\
for all $k\in[0,u_{\max}]$, for all $\xi\in\mathcal{C}^\infty([0,T)\times\mathbb R^\ell)$, $\xi\geq0$,\\[-16pt]
\begin{align}\label{content}
&\displaystyle\int_0^T\!\!\!\int_\Omega\displaystyle \left\{|\uOdt-k|\xi_t+sign(\uOdt-k)\Bigl[f(\uOdt)-f(k)-\nabla_{\!\mathcal O}\phi(\uOdt)\Bigr].\nabla\xi\displaystyle\right\}dxdt\nonumber\\[-2pt]
&+\displaystyle\int_0^T\!\!\!\int_{\partial\Omega} \left|f(k).\eta(x)\right|\xi(t,x) d{\mathcal{H}}^{\ell-1}(x)dt+\displaystyle\int_\Omega |u_0-k|\xi(0,x)dx\geq -\upsilon_{\mathcal{O},\delta t}(\xi),
\end{align}

\vspace*{-3pt}\noindent
where\;\; $\forall \xi\in\mathcal{C}^\infty([0,T)\times\mathbb R^\ell)$, $\upsilon_{\mathcal{O},\delta t}(\xi)\rightarrow0$ when $h\rightarrow 0$.
\end{theorem}
In order to pass to the limit in \eqref{content} using only the $L^\infty$ bound on $\uOdt$, one can adapt the notion of an entropy-process solution to problem $(P)$ in the entropy sense \eqref{ESP}.
\begin{definition}\label{entrprosol}
Let $\mu\in L^\infty(Q\times(0,1))$. The function $\mu=\mu(t,x,\alpha)$ is called an entropy-process solution to the problem $(P)$ if $\forall k\in  [0,u_{\max}]$,  $\forall\xi\in \mathcal{C}^\infty([0,T)\times\mathbb R^\ell)$, with $\xi\geq 0$, the following inequalities  hold:\\[-16pt]
\begin{eqnarray*}\label{ESP1}
&\displaystyle\int_0^T\int_\Omega\int_0^1\displaystyle \left\{|\mu-k|\xi_t+sign(\mu-k)\Bigl[f(\mu)-f(k)\Bigr].\nabla\xi\displaystyle\right\}dxdtd\alpha\nonumber\\[-4pt]
&-\displaystyle\int_0^T\int_\Omega\nabla|\phi(u)-\phi(k)|.\nabla\xi dxdt+\displaystyle\int_0^T\int_{\partial\Omega} \left|f(k).\eta(x)\right|\xi(t,x) d{\mathcal{H}}^{\ell-1}(x)dt\nonumber\\[-4pt]
&+\displaystyle\int_\Omega |u_0-k|\xi(0,x)dx\geq 0,
\;\;\;\text{where\; $u(t,x):=\displaystyle\int_0^1\mu(t,x,\alpha)d\alpha$}.
\end{eqnarray*}
\end{definition}
 From Theorem~\ref{th:approx-ineq} we derive the following result which, however, will not be conclusive. In the sequel, we will upgrade (or circumvent, see Remark~\ref{rem:direct-IPS}) this claim.
\begin{proposition}\label{conventropyprocess}
Let $\uOdt$ be the approximate solution of the problem $(P)$ defined by  \eqref{esti0}, \eqref{esti1}. There exists an entropy-process solution $\mu$ of $(P)$ in the sense of Definition \ref{entrprosol} and a subsequence of $(\uOdt)_{\mathcal{O},\delta t}$ , such that:\\[1pt]
$\bullet$ The sequence $(\uOdt)_{\mathcal{O},\delta t}$ converges to $\mu$ in the nonlinear weak-$*$ sense.\\
$\bullet$
 Moreover, $(\phi(\uOdt))_{\mathcal{O},\delta t}$ converges strongly in $L^2(Q)$ to $\phi(u)$,
$u=\int_0^1\mu(t,x,\alpha)d\alpha$,\\
\hspace*{6pt} and
$(\nabla_{\mathcal{O}}\phi(\uOdt))_{\mathcal{O},\delta t}\rightharpoonup\nabla \phi(u)$ in $(L^2(Q))^\ell$ weakly,
as $h,\delta t\to 0$.
\end{proposition}
\begin{proof}\smartqed The proof is essentially the same as in main reference papers dealing with finite volume scheme for degenerate parabolic equations (see \cite{EGHMichel,ABK}).
\qed
\end{proof}

\vspace*{-5pt}\section{Reduction of entropy-process solution: semigroup arguments}\label{sec:4}
In the context of the Dirichlet problem (see \cite{eymard2000finite,EGHMichel}) there holds the uniqueness and reduction result
stating that an entropy-process solution $\mu$ is $\alpha$-independent, so that it reduces to an entropy solution.
The lack of regularity of the fluxes at the boundary makes it difficult to prove the analogous result with zero-flux conditions. Here, we show how this difficulty can be bypassed, using classical tools and a new notion of \emph{integral-process solution} in the abstract context of nonlinear semigroup theory (\cite{BCP}).

\vspace{-5pt}\subsection{Notion of integral-process solution and equivalence result}
Given a Banach space $X$ and  an accretive operator $A\subset X\times X$,  $u\in C([0,T];X)$ is called integral solution (see B\'enilan et al.~\cite{BCP,BarthelemyBenilan}) of the abstract evolution problem \eqref{eq:Abstr-Pb}
if, $\|\cdot\|$ being the norm and $[u,v]:=\lim_{\lambda\downarrow 0} \frac{\|u+\lambda v\|-\|u\|}\lambda$ the bracket on $X$, one has $u(0)=u_0$ and the following family of inequalities holds:
\begin{equation*}\label{eq:integral-sol}
\forall (\hat u,\hat z)\in A \;\;\;\|u(t)\!-\!\hat u\| - \|u(s)\!-\!\hat u\|\leq \int_s^t [u(\tau)\!-\!\hat u,h(\tau)\!-\!\hat z], \;\;\; 0\leq s\leq t\leq T.
\end{equation*}
For $m$-accretive operators the classical in the nonlinear semigroup theory notion of mild solution coincides with the notion of integral solution, so that we have
\begin{proposition}
 Assume that $A$ is $m$-accretive, with $\overline{\text{Dom}(A)}^{\|\cdot\|_X}=X$.
 Then for any $h\in L^1((0,T);X)$, $u_0 \in X$ there exists a unique integral solution of
 \eqref{eq:Abstr-Pb}.
 \end{proposition}
We refer to \cite{BCP} for the proof of uniqueness of an integral solution and to \cite{BarthelemyBenilan} for
a generalization relevant to our case: continuity of $u:[0,T]\to X$ can be relaxed, cf.~\eqref{processunitial}. We propose a variant of the above notion that we call \emph{integral-process solution}. This notion is motivated by an application in the setting where $X$ is a Lebesgue space on $\Omega\subset{\mathbb R}^\ell$ and $\nu$ is a \emph{nonlinear weak-$*$ limit} (see~\cite{eymard2000finite}) of approximate solutions.
\begin{definition}
Let  $A$ be an accretive operator on $X$, $h\in L^1(0,T;X)$ and $u_0\in X$. An $X$-valued function $\nu$ of $(t,\alpha)\in [0,T]\times [0,1]$ is an integral-process solution of abstract problem $u'+Au\ni h$ on $[0,T]$ with datum $\nu(0,\cdot,\alpha)\equiv u_0(\cdot)$, if 
 for all $(\hat{u},\hat z)\in A$
\begin{align}\label{intprocsolu}
\displaystyle\int_0^1\!\!\Bigl(\|\nu(t,\alpha)\!-\hat{u}\|-\|\nu(s,\alpha)\!-\hat{u}\|\Bigr)d\alpha\!
\leq\displaystyle\int_0^1\!\!\!\!\int_s^t\Bigl[v(\tau,\alpha)-\hat{u}, h(\tau)- \hat z\Bigr]d\tau d\alpha
\end{align}
for $0< s\leq t\leq T$  and the initial condition is satisfied in the sense
\begin{align}\label{processunitial}
\mbox{ess-}\lim\nolimits_{t\downarrow0} \textstyle{\int_0^1} \|\nu(t,\alpha)- u_0\|d\alpha=0.
\end{align}
\end{definition}
%
The main fact concerning integral-process solutions is the following result (\cite{mathese}).
 \begin{theorem}\label{th:Integral-Proc}
Assume that $A$ is $m$-accretive in $X$ and $u_0\in\overline{D(A)}$. Then $\nu$ is an integral-process solution of \eqref{eq:Abstr-Pb} if and only if $\nu$ is independent on $\alpha$ and for all $\alpha$, $\nu(.,\alpha)$ coincides with the unique integral and mild solution $u(\cdot)$ of \eqref{eq:Abstr-Pb}.
\end{theorem}

\vspace*{-5pt}\subsection{Convergence of the scheme}
Let us define the operator $ A_{f,\phi}$ on $L^1(\Omega;[0,u_{\max}])\subset X=L^1(\Omega)$ endowed with $\|\cdot\|_1$:
\begin{equation*}
(v,z)\!\in A_{f,\phi}=\!\left \{\begin{array}{ll}
 \!\! v \mbox{ such that  } v \mbox{ is a \emph{trace regular} solution of } (S), \mbox{ with  } g=v+z
\end{array}
\right\}
\end{equation*}
(instead of $L^1(\Omega)$ we can work in $L^1(\Omega;[0,u_{\max}])$ due to the confinement principle for solutions of $(S)$).
The main result of this paper is the following theorem.
\begin{theorem}\label{unicite}
   Assume operator $A_{f,\phi}$ on $L^1(\Omega;[0,u_{\max}])$ is $m$-accretive densely defined, then any entropy-process-solution of $(P)$ is its unique entropy solution. In particular, the scheme \eqref{esti0},\eqref{esti1} for discretization of $(P)$ in the sense \eqref{ESP} is convergent:
\begin{equation*}
\forall p\in[1,+\infty)\;\;\; \uOdt\longrightarrow u \; \mbox{ in } L^p(0,T\times\Omega) \;\mbox{ as }\; \max(\delta t, h)\longrightarrow 0.
\end{equation*}
\end{theorem}
\begin{proof}\smartqed
   First, in Proposition~\ref{conventropyprocess} we prove that the approximate solutions $\uOdt$ converge towards an entropy-process solution $\mu$. Then, with the technique of \cite{BF,BG} we compare the entropy-process solution $\mu$ and a trace-regular solution $\hat u$ of stationary problem $(S)$.  We find that $\mu$ is also an integral-process solution. By Theorem~\ref{th:Integral-Proc}, $\mu$ is independent on the variable $\alpha$. Therefore $\mu(\cdot,\alpha)$ coincides with the unique integral solution of the abstract evolution problem \eqref{eq:Abstr-Pb} governed by operator $A_{f,\phi}$; we know from the analysis of \cite{BF,BG} that it is also the unique entropy solution of $(P)$.
  \qed
\end{proof}
 Theorem~\ref{unicite} is applicable in the following three cases where trace-regularity for the solutions
 of $(S)$ can be justified, at least for a dense set of source terms.
\begin{proposition}\label{prop:hyperbolic}
Assume that  $\ell\geq1$, and $u_c=u_{\max}$ (i.e., $(P)$ is purely hyperbolic). Then $A_{f,\phi}$ is $m$-accretive densely defined on  $L^1(\Omega;[0,u_{\max}])$.
\end{proposition}
\begin{proposition}\label{prop:parabolic}
Assume that $\ell\geq1$ and $u_c=0$ (i.e. $(P)$ is non-degenerate parabolic). Then $A_{f,\phi}$ is $m$-accretive densely defined on  $L^1(\Omega;[0,u_{\max}])$
if $f\circ\phi^{-1}\!\in\mathcal{C}^{0,\gamma},\gamma>0$.
\end{proposition}
\begin{proposition}\label{prop:degenerate1D}
Assume that  $\Omega=(a,b)$ (thus, $\ell=1$). Then $A_{f,\phi}$ is $m$-accretive densely defined on  $L^1(\Omega;[0,u_{\max}])$.
\end{proposition}
Prop.~\ref{prop:hyperbolic} follows by the strong trace results of \cite{VAS,PAN1} (cf. \cite{BFK}), Prop.~\ref{prop:parabolic} is justified like in \cite{BF}, while Prop.~\ref{prop:degenerate1D} was an ingredient of the uniqueness proof in \cite{BG}.
%
%
%

\begin{remark}\label{rem:direct-IPS}
Actually, the use of entropy-process solutions can be circumvented. Observe that the stationary problem $(S)$ can be discretized
with the scheme analogous to the time-implicit scheme used for the evolution problem $(P)$.
 Consider the situation where strong compactness (and convergence to $\hat u\in \text{Dom}(A_{f,\phi})$) can be proved for approximate solutions $\hat u_{\mathcal O}$ of $(S)$ but only nonlinear weak-$*$ compactness for approximate solutions $\uOdt$ of $(P)$ is known (this occurs when $\ell=1$, where compactness of $\hat u_{\mathcal O}(x_i)$, for all $x_i\in \mathbb{Q}$, is immediate: see the arguments developed in \cite{A-SpringerProc}).
 Then convergence of the stationary scheme is easily proved, moreover, one infers inequalities \eqref{intprocsolu} for the limit $\nu(\cdot,\alpha)$ of $\uOdt$. Then, the result of Theorem~\ref{th:Integral-Proc}
 proves convergence of the scheme for the evolution problem. In a future work, this argument will
 be applied to a large variety of one-dimensional degenerate parabolic conservation laws with boundary
 conditions or interface coupling conditions.
\end{remark}

\vspace*{-5pt} \section{Numerical experiments}
 We conclude with $1D$ numerical illustrations  presented in Fig.~\ref{hfigure}(a),(c),
 obtained with the explicit analogue of the scheme~\eqref{esti0},\eqref{esti1} under the \emph{ad hoc} CFL restrictions. On this occasion, we use the scheme to highlight the importance of hypothesis \eqref{fmax}. In the test of Fig.~\ref{hfigure}(b) assumption \eqref{fmax} fails, and a boundary  layer appears. If one refines the mesh one observes convergence of $u_ {\mathcal O_h, \delta t_h}$ towards a function bounded by $\| u_0 \|_\infty$ while the sequence $(u_ {\mathcal O_h, \delta t_h})_h$ seems unbounded.
   However, the condition of zero flux imposed in \eqref{esti1} is relaxed in the limit, making formulation \eqref{ESP} inappropriate outside the framework \eqref{fmax}. Introduction of appropriate boundary formulation
   satisfied by the limit of the scheme, in absence of \eqref{fmax}, is postponed to future work.
\begin{figure}[!ht]
  \vspace{-42pt}\hspace*{-10pt}
       \begin{minipage}[b]{0.3\linewidth}
          \centering \scalebox{0.2}{\includegraphics{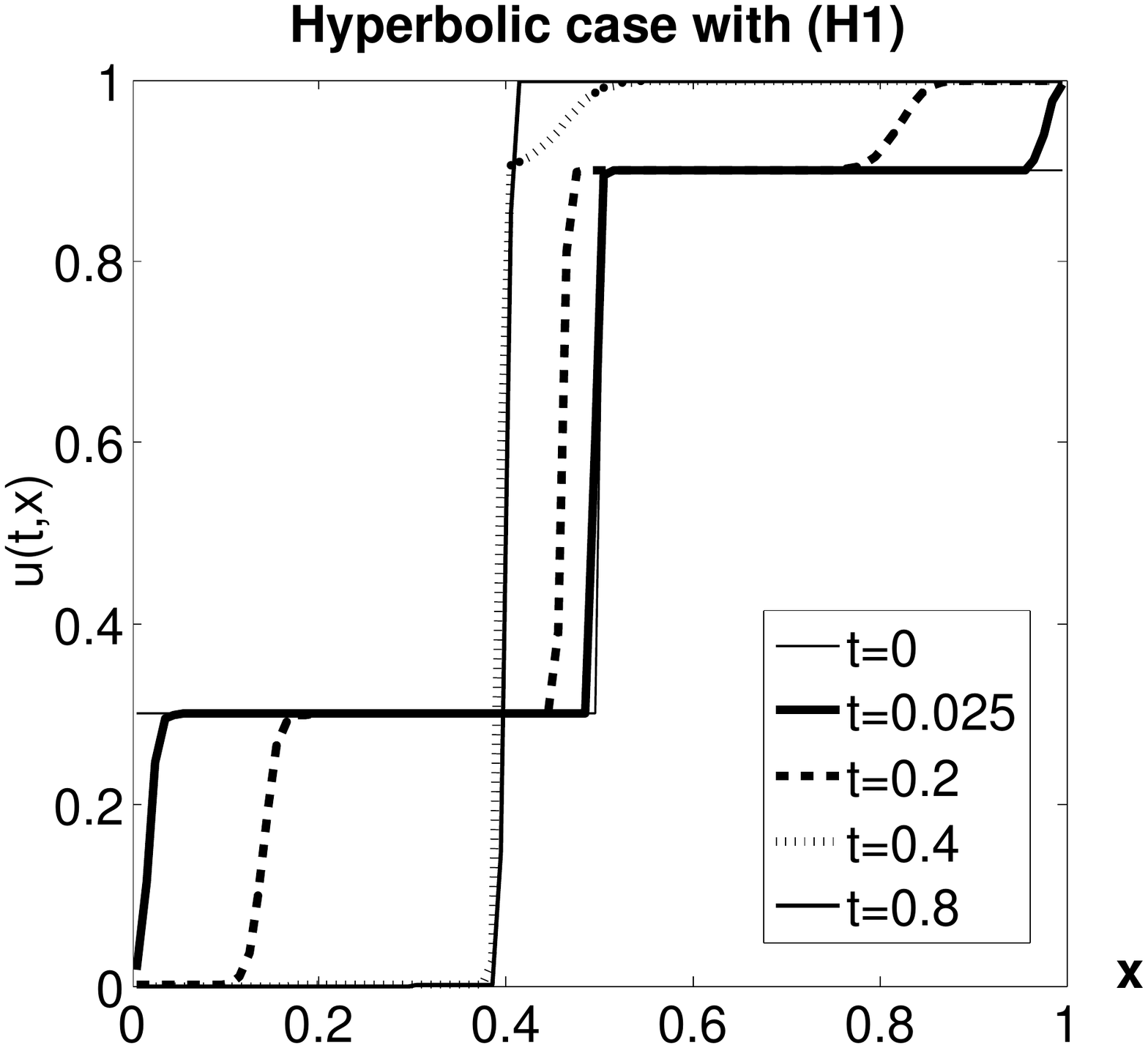}}
       \end{minipage}\hfill
      \begin{minipage}[b]{0.3\linewidth}
       \centering \scalebox{0.2}{\includegraphics{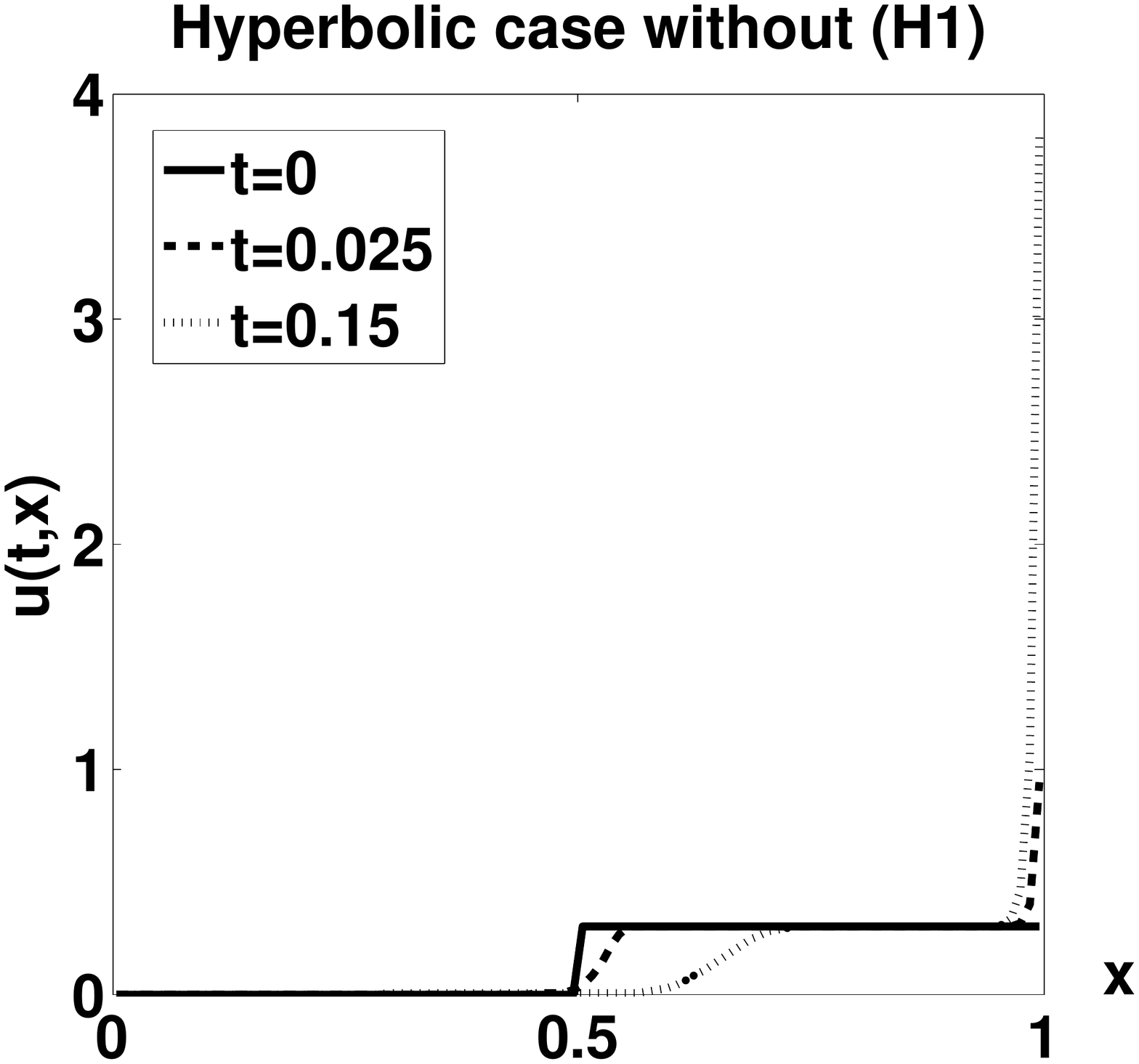}}
       \end{minipage}\hfill
       \begin{minipage}[b]{0.30\linewidth}
          \centering \scalebox{0.2}{\includegraphics{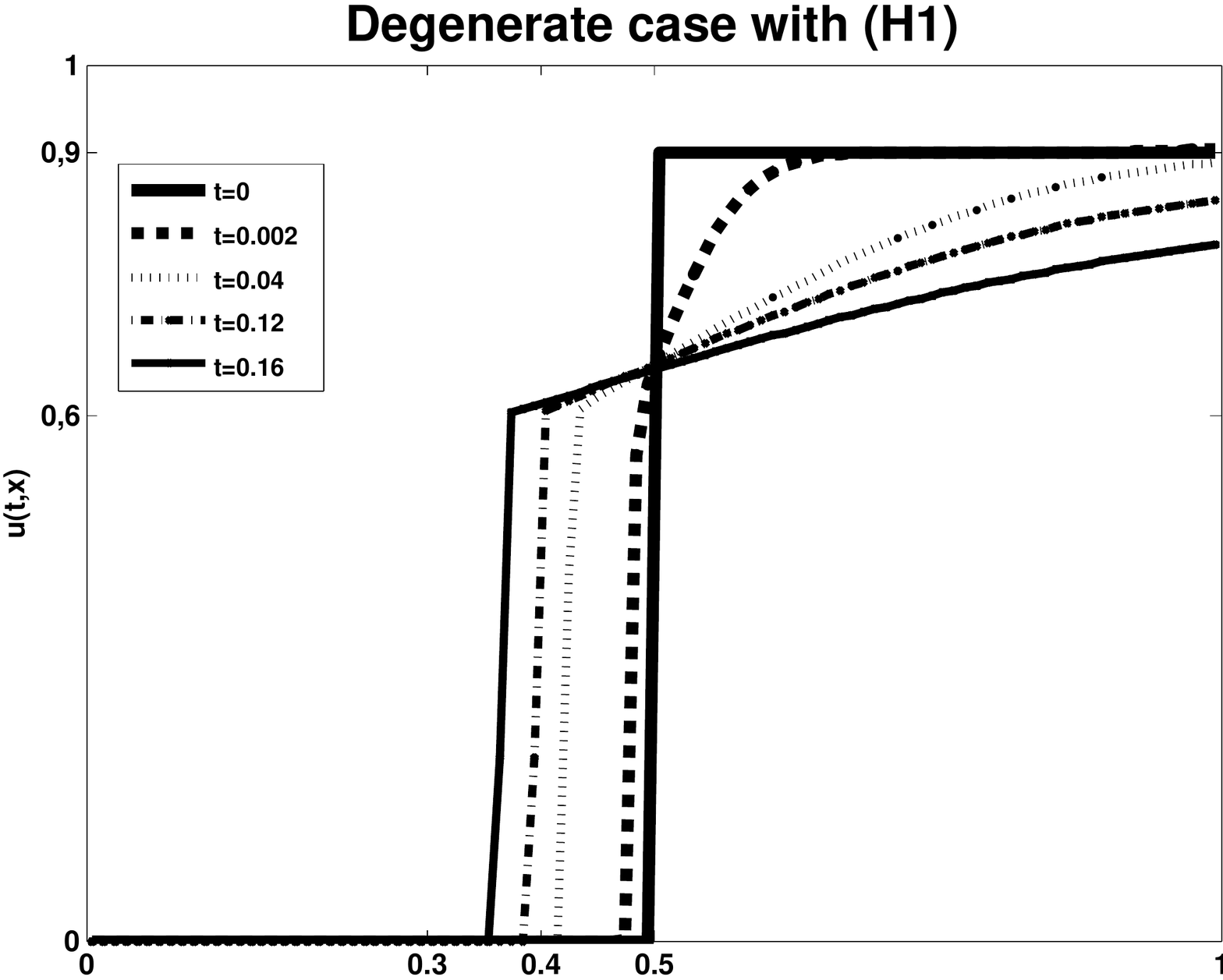}}
                 \end{minipage}\hfill
       \vspace{-35pt}
\caption{
(a)  $f(u)=u(1-u)$, $\phi\equiv 0$\;\;\;(b) \;$f(u)=\frac{u^2}{2}$, $\phi\equiv 0$\;\; \;(c)\; $f(u)=u(1-u)$, $\phi(u)=(u-0.6)^+$.}
\label{hfigure}
    \end{figure}


\vspace*{-22pt}\begin{acknowledgement}
This work has been supported by the French ANR project CoToCoLa.
\end{acknowledgement}

\vspace*{-15pt}

\end{document}